\documentclass[12pt,reqno]{amsart}
\usepackage{amsmath,amsthm,amssymb,amsfonts,amscd}
\usepackage{mathrsfs}%mathtools
\usepackage{bbm}
\usepackage{bbding}

\usepackage{bm}
\usepackage[backref=page]{hyperref}

\usepackage{geometry}
\usepackage{color}
\usepackage{xcolor}
\usepackage{cite}
\usepackage{picture,epic}
\usepackage{tikz}

\numberwithin{equation}{section}

\theoremstyle{plain}
\newtheorem{theorem}{Theorem}[section]
\newtheorem{lemma}[theorem]{Lemma}

\newtheorem{proposition}[theorem]{Proposition}

\theoremstyle{definition}
\theoremstyle{remark}

\renewcommand{\Re}{\operatorname{Re}}

\newcommand{\sym}{\operatorname{sym}}

\newcommand{\GL}{\operatorname{GL}}
\newcommand{\SL}{\operatorname{SL}}

\makeatletter
\def\@tocline#1#2#3#4#5#6#7{\relax
  \ifnum #1>\c@tocdepth % then omit
  \else
    \par \addpenalty\@secpenalty\addvspace{#2}%
    \begingroup \hyphenpenalty\@M
    \@ifempty{#4}{%
      \@tempdima\csname r@tocindent\number#1\endcsname\relax
    }{%
      \@tempdima#4\relax
    }%
    \parindent\z@ \leftskip#3\relax \advance\leftskip\@tempdima\relax
    \rightskip\@pnumwidth plus4em \parfillskip-\@pnumwidth
    #5\leavevmode\hskip-\@tempdima
      \ifcase #1
       \or\or \hskip 1em \or \hskip 2em \else \hskip 3em \fi%
      #6\nobreak\relax
    \hfill\hbox to\@pnumwidth{\@tocpagenum{#7}}\par% <---- \dotfill -> \hfill
    \nobreak
    \endgroup
  \fi}
\makeatother

\begin{document}

\title[Lower bounds for moments]
{Lower bounds for moments of quadratic twisted self-dual $\GL(3)$ central $L$-values}
% of a  self-dual cusp form

\author{Shenghao Hua and Bingrong Huang}

\address{Data Science Institute and School of Mathematics \\ Shandong University \\ Jinan \\ Shandong 250100 \\China}
\email{huashenghao@vip.qq.com}
\email{brhuang@sdu.edu.cn}

\date{\today}

\begin{abstract}
In this paper, we prove the conjectured order lower bound for the $k$-th moment of central values of quadratic twisted self-dual $\GL(3)$ $L$-functions for all $k\geq 1$, based on our recent work on the twisted first moment of central values in this family of $L$-functions.
\end{abstract}


\maketitle
%\setcounter{tocdepth}{1}%使目录只显示到节
%\tableofcontents

%%%%%%%%%%%%%%%%%%%%%%%%%%%%%%%%%%%%%%%%%%%%%%%%%%%%%%%%%%%%%%%%
%%%%%                        Section                       %%%%%
%%%%%%%%%%%%%%%%%%%%%%%%%%%%%%%%%%%%%%%%%%%%%%%%%%%%%%%%%%%%%%%%
\section{Introduction} \label{sec:Intr}

Estimating moments of central values in families of $L$-functions  is an active topic of number theory.
There are now well-established conjectures for these moments,
thanks to the work of
Keating and Snaith \cite{KS00a,KS00b},
with subsequently works of
Diaconu, Goldfeld and Hoffstein \cite{DGH03} and Conrey, Farmer, Keating, Rubinstein and Snaith \cite{CFKRS05}.
%Katz and Sarnak\cite{KS99} were shown there are four cases(also could say three cases when omit parity of special orthogonal group $\SO(n)$'s degree) of guessed moments for different $L$-function families corresponding to four classes of classical groups for some deeper reason in random matrix theory.
%Conditional on the Birch-Swinnerton-Dyer Conjecture(BSD Conjecture),
%the work of Soundararajan\cite{RS15} establishes the conjectured upper bound for moments of central $L$-values of quadratic of elliptic,
%and they also achieved the conjectured lower bound for moments without needs of the BSD Conjecture.
In \cite{RS05,RS06}, Rudnick and Soundararajan developed a  method for establishing conjectured order lower bounds for moments of central values of families of $L$-functions, provided a little more than the first moment of this family of $L$-functions can be computed.
This method was extended by Radziwi{\l\l} and Soundararajan \cite{RS13} and Heap and  Soundararajan \cite{HS20} to all positive real moments.
Recently, the authors \cite{HH21} proved an asymptotic formula for the twisted first moment of central values of quadratic twisted $\GL(3)$ $L$-functions.
In this paper, we work out the conjectured order lower bound for the $k$-th moment of central values of this family of $L$-functions for all $k\geq 1$, based on \cite{HS20} and \cite{HH21}.

Let $\phi$ be a self-dual Hecke--Maass cusp form of type $(\nu,\nu)$ for $\SL(3,\mathbb{Z})$ with the normalized Fourier coefficients
 $A(m,n)$.
We have the conjugation relation $A(m,n)=\overline{A(n,m)}=A(n,m)$, see Goldfeld \cite[Theorem 9.3.11]{GB06}.
For an automorphic representation $\pi$ of $\GL(3,\mathbb{A_Q})$,
the symmetric square lift $L$-function $L(s,\sym^2\pi)$ has a simple pole at $s=1$ if and only if $\pi$ is the Gelbart--Jacquet lift  of an automorphic representation on $\GL(2,\mathbb{A_Q})$ with trivial central character, see \cite{GRS99}. That is, $\pi$ is a self-contragredient cuspidal automorphic representation, see \cite{GJ76} and \cite{RM14}.
So $\phi$ is self-dual if and only if $L(s,\sym^2 \phi)$ has a simple pole at $s=1$. Note that when the $\GL(3)$ cusp form is self-dual, its Fourier coefficients are all real.
We denote $\theta_3$ be the least common upper bound of power of $p$ for $|A(p,1)|$, i.e., $|A(p,1)|\leq 3p^{\theta_3}$ for all prime $p$.
The Generalized Ramanujan Conjecture implies that $\theta_3=0$, and from Kim--Sarnak \cite[Appendix 2]{Kim2003} we know $\theta_3\leq\frac{5}{14}$.

Any real primitive character %to the modulus $q$
must be of the form $\chi_d(n)=(\frac{d}{n})$ where $d$ is a fundamental discriminant \cite[Theorem 9.13]{MV07}, i.e., a product of pairwise coprime integers of the form $-4$, $\pm 8$, $(-1)^{\frac{p-1}{2}}p$ where $p$ is an odd prime.
There are two primitive characters to the modulus $q$ if $8\parallel q$ and only one otherwise.
For $\Re s$ sufficiently large we define the twisted $L$-function $L(s,\phi\otimes \chi_{8d})$ to be
$$L(s,\phi\otimes \chi_{8d})=\sum_{n=1}^\infty \frac{A(n,1)\chi_{8d}(n)}{n^s}.$$
It has an analytic continuation to the entire complex plane and satisfies the functional equation (see \cite[Theorem 7.1.3]{GB06})
\begin{equation*}
  \Lambda (s,\phi\otimes \chi_{8d}) := (8d)^{s/2}\pi^ {-3s/2}\prod_{i=1}^3\Gamma\left(\frac{s-\gamma_i}{2}\right)L(s,\phi\otimes \chi_{8d})
  =  \Lambda (1-s,\phi\otimes \chi_{8d})
%\pi^{-3(1-s)/2}\prod_{i=1}^3\Gamma(\frac{1-s
%+\gamma_i}{2})L(1-s,\phi
%\otimes \chi_{8d})
\end{equation*}
when $d$ is positive and $\phi$ is self-dual, where $\gamma_1=1-3\nu$, $\gamma_2=0$, and $\gamma_3=-1+3\nu$ are the Langlands parameters of $\phi$.

Recently, we \cite{HH21} proved asymptotic formulas for the twisted first moment of quadratic twisted $\GL(3)$ central $L$-values.
In this paper,
we only consider the self-dual case and prove the conjectured lower bound for the $k$-th moment for quadratic twisted self-dual $\GL(3)$ central $L$-values for all $k\geq 1$, by using our twisted first moment result.
We will use the lower bounds principle of Heap and Soundararajan  \cite{HS20}, where
they achieved lower bounds of moments of Riemann zeta function on the critical line,
which is a case touch to unitary group.
Recently, Gao and Zhao \cite{GZ21} estiblished the analogous result on lower bounds for moments of central values of quadratic twisted modular form $L$-functions.
%by using of the twisted first moment of quadratic twisted modular functions result of Shen \cite{Sh21}.
We extend their method to the $\GL(3)$ case.
Our main result is as follows.

\begin{theorem}\label{thm1}
Let $\phi$ be a self-dual Hecke--Maass cusp form of $\SL(3,\mathbb{Z})$, and $\Phi$ a smooth nonnegative Schwarz class function supported in interval $(1,2)$.
For any $k\geq 1$ we have
\begin{equation}\label{eq1}
  \mathop{\sum\nolimits^\flat}_{2\nmid d} \left| L\left(\frac{1}{2},\phi\otimes \chi_{8d}\right)\right|^k \Phi\left(\frac{d}{X}\right)
  \gg_{k,\Phi} X(\log X)^{\frac{k(k+1)}{2}},
\end{equation}
where $\sum_{2\nmid d}\nolimits^\flat$ means summing over positive odd square-free $d$.
\end{theorem}

%We can see its magnitude corresponding to special sympletic group $\SP(n)$ because it's biggest.

Our result shows that the family of quadratic twisted $L$-functions $L(s,\phi\otimes \chi_{8d})$ with $\phi$ self-dual is a symplectic family, which is the same as the family of  quadratic Dirichlet $L$-functions $L(s,\chi_{8d})$.
Note that the family of quadratic twists of a given $\GL(2)$ newform $f$, $L(s, f\otimes \chi_{8d})$,  is an orthogonal family as proved in \cite{RS05,RS06}.
If $\phi$ is a non self-dual Hecke--Maass cusp form of $\SL(3,\mathbb{Z})$, then our result in \cite[Theorem 1.3]{HH21} implies that the  family of quadratic twisted $L$-functions $L(s,\phi\otimes \chi_{8d})$  is an orthogonal family as the $\GL(2)$ case.

\section{Preliminaries}

Let $N,M$  be sufficiently large natural numbers depending on $k$ only.
Let $\{\ell_j\}|_{j=1}^R$ be a sequence of decreasing even integers, where $\ell_1= 2\lceil N \log \log X\rceil$ and $\ell_{j+1} = 2 \lceil N \log \ell_j \rceil$ for $j\geq 1$, and $R$ is the largest integer satisfying $\ell_R >10^M$.
Assume that $M$  is large enough so that $\ell_{j-1} >\ell_{j}^2$ for all $2 \leq j \leq R$.
We will use the following two inequalities
\begin{equation*}
R \ll \log \log \ell_1 \quad \mbox{and} \quad \sum^R_{j=1}\frac 1{\ell_j} \leq \frac 2{\ell_R}.
\end{equation*}

Let ${P}_1$ be the set of primes in interval $[3, X^{1/\ell_1^2}]$ and
${P_j}$ be the set of primes in interval $(X^{1/\ell_{j-1}^2}, X^{1/\ell_j^2}]$ for $2\le j\le R$.
We define
\begin{equation*}
{\mathcal P}_j(d)=\sum_{p\in P_j}\frac{A(p,1)}{\sqrt{p}}\chi_{8d}(p),
\end{equation*}
$E_{\ell}(x)=\sum_{j=0}^{\ell} \frac{x^{j}}{j!}$,
${\mathcal N}_j(d,\alpha)=E_{\ell_j}(\alpha{\mathcal P}_j(d))$,
and
\begin{equation*}
\mathcal{N}(d,\alpha)=\prod_{j=1}^{R}
{\mathcal N}_j(d,\alpha)
=\prod_{j=1}^{R}E_{\ell_j}(\alpha{\mathcal P}_j(d)), \quad(\alpha\in \mathbb{R}).
\end{equation*}
Note that ${\mathcal P}_j(d)$
are all real for self-dual form $\phi$, since $A(p,1)\in\mathbb{R}$.
For $l> 0$ even and $x\in \mathbb{R}$,
then we can choose $x_0\neq 0$ which is a local minimum of $E_l$ because $E_l(\pm \infty)=+\infty$ and $E_l'(0)=E_{l-1}(0)=1$,
thus $E_l'(x_0)=E_{l-1}(x_0)=0$,
and we know $E_l(x_0)=\frac{x_0^l}{l!}+E_{l-1}(x_0)>0$.
Together with $E_l(0)=1$,
we have $E_l(x)$ must be positive for any even integer $l\geq 0$ and $x\in \mathbb{R}$.
This is a part of \cite[Lemma 1]{RS15}.
Hence we know ${\mathcal N}_j(d,\alpha)$ is always positive in our case.

The case $k=1$ in  Theorem \ref{thm1} is a simple consequence of our Lemma \ref{lm:mysum} below with $l=1$ which we proved in \cite{HH21}. From now on, we assume $k>1$.
To prove Theorem \ref{thm1}, we need the following propositions.
By the H\"older's inequality, we know
\begin{multline*}
\sum_{2\nmid d}\nolimits^\flat L(\frac{1}{2},\phi\otimes\chi_{8d}){\mathcal N}(d,k-1)\Phi(\frac{d}{X}) \\
\leq(\sum_{2\nmid d}\nolimits^\flat |L(\frac{1}{2},\phi\otimes \chi_{8d})|^{k}\Phi(\frac{d}{X}))
^{\frac{1}{k}}(\sum_{2\nmid d}\nolimits^\flat {\mathcal N}(d,k-1)^{\frac{k}{k-1}}\Phi(\frac{d}{X}))
^{\frac{k-1}{k}}
\end{multline*}
when $k>1$. So by Propositions \ref{pp1} and \ref{pp2} we prove Theorem \ref{thm1}.

\begin{proposition}\label{pp1}
When $k>1$, we have
\begin{equation*}
\sum_{2\nmid d}\nolimits^\flat L(\frac{1}{2},\phi\otimes \chi_{8d}){\mathcal N}(d,k-1)\Phi(\frac{d}{X})\gg X(\log X)^{\frac{k^2+1}{2}}.
\end{equation*}
\end{proposition}

\begin{proposition}\label{pp2}
When $k>1$, we have
\begin{equation*}
\sum_{2\nmid d}\nolimits^\flat {\mathcal N}(d,k-1)^{\frac{k}{k-1}}\Phi(\frac{d}{X})
\ll X(\log X)^{\frac{k^2}{2}}.
\end{equation*}
\end{proposition}

To prove the above two propositions we will use the following Lemma.
\begin{lemma}[\cite{LWY05}]\label{lm:sum}
For normalized self-dual Hecke--Maass cusp form $\phi$ for $\SL(3,\mathbb{Z})$ with Fourier coefficients $A(m,n)$,
we have
\begin{equation*}
\sum_{p\leq x}\frac{A(p,1)^2}{p}=\log\log x+O_\phi(1).
\end{equation*}
\end{lemma}

\section{Proof of Proposition \ref{pp1}}

For convenience we write an integer $n=(n)_1(n)_2^2$ with $(n)_1$ is square-free.
We would prove Proposition \ref{pp1} with the help  of the following lemma.
\begin{lemma}[\cite{HH21}]\label{lm:mysum}
Let $\phi$ be a self-dual Hecke--Maass cusp form of $\operatorname{SL}(3,\mathbb{Z})$ with normalized Fourier coefficients $A(m,n)$.
For sufficiently large $X>0$ and odd $l\ll X^{\frac{1}{10}-\varepsilon}$ with arbitrarily small $\varepsilon >0$, and any smooth nonnegative Schwarz class function $\Phi$ supported in the interval $(1,2)$, we have
\begin{multline*}
\sum_{2\nmid d}\nolimits^\flat L(\frac{1}{2},\phi\otimes \chi_{8d})\chi_{8d}(l)\Phi(\frac{d}{X})
= c\check{\Phi}(0)\frac{\lim_{s\to 1} (s-1)L^{\{2\}}(s,\sym^2 \phi)}
{\sqrt{(l)_1}}
\\
\times X(G(l)(\log \frac{X}{(l)_1 ^{\frac{2}{3}}}+c_1)+H(l))
+O_{\Phi}(l^{\frac{3}{4}+\varepsilon}
X^{\frac{19}{20}+\varepsilon})
\end{multline*}
where $\hat{\Phi}(x)=\int_{-\infty}^\infty \Phi(t)e(-xt)dt$, $\Phi_{(3)}=\max_{0\leq j \leq 3}\int_1^2|\Phi^{(j)}(t)|dt$,
%$\sum_{2\nmid d}\nolimits^\flat$ means sum over positive odd square-free $d$,
$$G(l)=\prod_{\textrm{odd prime }p}G_p(l)$$
with
\begin{equation*}
  G_p(p^a M)=G_p(p^a)=\left\{\begin{aligned}
  &1+O(p_2^{-2+2\theta_3+\varepsilon}),\quad &2\mid a,\\
  &A(p_1,1)+O(p^{-1+2\theta_3+\varepsilon}),\quad
  &2\nmid a,
  \end{aligned}
\right.
\end{equation*}
for $(M,p)=1$, with effective $O$-constant no more than 35.
$$H(l)=O\Big( \prod_{p|l_1}
     (|A(p_1,1)|+p^{-1+2\theta_3+\varepsilon})
     +\sum_{p\mid l_1}
     p^{-1+2\theta_3}\log p
     \prod_{\substack{
     p_1|l_1\\ p_1\neq p}}
     (|A(p_1,1)|+p^{-1+2\theta_3+\varepsilon})  \Big)
     ,$$
with effective $O$-constant.
And $c,c_1$ are constants depending on $\phi$ and $\Phi$.
\end{lemma}

Let $g(n)$ be a multiplicative function with $g(p^m)=m!$.
Let $\Omega(n)$ be the total number of prime factors of  $n$.
Let indicator function $b_j(n)=1$ when all prime factors of $n$ are in $P_j$ and $\Omega(n)\leq \ell_j$, and $b_j(n)=0$ otherwise.
Let $A(m)$ be a completely multiplicative function with $A(p)=A(p,1)$ when $p$ is prime,
so $A(m)=A(m,1)$ when $m$ is square-free.
Thus
\begin{equation*}
{\mathcal N}_j(d,k-1)= \sum_{n_j}\frac{A(n_j)}{\sqrt{n_j}} \frac{(k-1)^{\Omega(n_j)}}{g(n_j)}
b_j(n_j)\chi_{8d}(n_j)
\end{equation*}
for $1\leq j\leq R$.

From Lemma \ref{lm:mysum} we know
\begin{multline*}
\sum_{2\nmid d}\nolimits^\flat L(\frac{1}{2},\phi\otimes \chi_{8d}){\mathcal N}(d,k-1)\Phi(\frac{d}{X})\\
\gg X\sum_{n_1,\dots,n_R}\prod_{j=1}^R\frac{A(n_j)
b_j(n_j)} {\sqrt{n_j(n_j)_1}}\frac{(k-1)^{\Omega(n_j)}}{g(n_j)}\\
\times (G(\prod_{j=1}^R n_j)(\log \frac{X}{((n_1)_1\dots(n_R)_1)^{\frac{2}{3}}}
+c_1)+H(\prod_{j=1}^R n_j)).
\end{multline*}
Here we only concentrate on the terms involving $G(\prod_{j=1}^R n_j)\log \frac{X}{((n_1)_1\dots(n_R)_1)^{\frac{2}{3}}}$.
By using the same argument in the following discussion,
the terms involving $G(\prod_{j=1}^R n_j)c_1,H(n_1\dots n_R)$ contribute as
$O(X(\log X)^{\frac{k^2+1}{2}-1+\varepsilon})$,
then
\begin{equation*}
\begin{split}
&\sum_{2\nmid d}\nolimits^\flat L(\frac{1}{2},\phi\otimes \chi_{8d}){\mathcal N}(d,k-1)\Phi(\frac{d}{X})\\
&\hskip 30pt \gg X\log X\sum_{n_1,\dots,n_R}G(\prod_{j=1}^R n_j)\prod_{j=1}^R\frac
{A(n_j)b_j(n_j)} {\sqrt{n_j(n_j)_1}}\frac{(k-1)^{\Omega(n_j)}}
{g(n_j)}\\
&\hskip 60pt -\frac{2}{3}X\sum_{n_1,\dots,n_R}G(\prod_{j=1}^R n_j)
(\prod_{j=1}^R\frac{A(n_j)b_j(n_j)} {\sqrt{n_j(n_j)_1}}\frac{(k-1)^{\Omega(n_j)}}
{g(n_j)}\log (\prod_{j=1}^R(n_j)_1)\\
&\hskip 90pt +O(X(\log X)^{\frac{k^2+1}{2}-1+2\theta_3+\varepsilon})\\
&\hskip 30pt =S_1-\frac{2}{3}S_2
+O(X(\log X)^{\frac{k^2+1}{2}-1+2\theta_3+\varepsilon}).
\end{split}
\end{equation*}

Evidently all $n_j$ in $S_1$  satisfy that $\Omega(n_j)\leq \ell_j$.
Removing the restriction for $\Omega(n_j)$ we get
\begin{equation*}
\begin{split}
S_1\geq X\log X\prod_{j=1}^R\Big (\prod_{p_j\in P_j}
(&\sum_{i=0}^{\infty}\frac{A(p_j)^{2i}G_{p_j}
(p_j^{2i})} {p^{i}}\frac{(k-1)^{2i}}{(2i)!}
+\sum_{i=0}^{\infty}\frac{A(p_j)^{2i+1}G_{p_j}
(p_j^{2i+1})} {p^{i+1}}\frac{(k-1)^{2i+1}}{(2i+1)!
})\\
& -\sum_{n_j}\frac{|A(n_j)\prod_{p_j\in P_j}G_{p_j}(n_j)|}
{\sqrt{n_j(n_j)_1}}
\frac{(k-1)^{\Omega(n_j)}}{g(n_j)} 2^{\Omega(n_j)-\ell_j}\Big )\prod_{p>X^{\frac{1}{\ell_R^2}}}G_p(1).
\end{split}
\end{equation*}
Note that $2^{\Omega(n_j)-\ell_j}> 1$ when $\Omega(n_j)>\ell_j$.
Hence we have
\begin{equation*}
\begin{split}
S_1\geq &\frac{X\log X}{\zeta(2-2\theta_3-0.01)^{35}} \prod_{j=1}^R\Bigg(\prod_{p_j\in P_j}
\Big(1-(\frac{(k-1)^2}{2}+k-1)\frac{A(p_j)^2}{p_j}
\Big)
^{-1}(1+O(\frac{|A(p_j)|+1}{p_j^{2-2\theta_3}}))\\
&\hskip 30pt -2^{-\ell_j}\prod_{p_j\in P_j}\Big(1+\sum_{i=1}^{\infty}\frac{A(p_j)^{2i}
G_{p_j}(p_j^{2i})} {p_j^{i}}\frac{(k-1)^{2i}2^{2i}}{(2i)!}
\\
&\hskip 60pt
+\sum_{i=0}^{\infty}\frac{A(p_j)^{2i+1}G_{p_j}(p_j^{2i+1})} {p_j^{i+1}}\frac{(k-1)^{2i+1}2^{2i+1}}{(2i+1)!
}\Big)\Bigg )\\
&\geq \frac{X\log X}{\zeta(2-2\theta_3-0.01)^{35}} \prod_{j=1}^R\Bigg (\exp\Big(\frac{k^2-1}{2}
\sum_{p_j\in P_j}(\frac{A(p_j,1)^2}{p_j}+O(\frac{|A(p_j,1)|}
{p_j^{2-2\theta_3}}))\Big)\\
&\hskip 30pt -2^{-\ell_j}\exp\Big((2(k-1)^2+2(k-1))\sum_{p_j\in P_j}(\frac{A(p_j,1)^2}{p_j}+O(\frac{|A(p_j,1)|}
{p_j^{2-2\theta_3}}))\Big) \Bigg ),
\end{split}
\end{equation*}
because $1+x\leq \exp(x)$ for $x\in\mathbb{R}$.
In our setting $M$ is large enough such that any $\ell_j$ is large,
and $N$ is large enough such that
\begin{equation}\label{doueq}
\frac{\ell_j}{4N}\leq \sum_{p_j\in P_j}\frac{A(p_j,1)^2}{p_j}\leq \frac{2\ell_j}{N}
\end{equation}
for every $j$ from Lemma \ref{lm:sum}.
Thus
\begin{multline*}
2^{-\ell_j}\exp\Big((2(k-1)^2+2(k-1))\sum_{p_j\in P_j}(\frac{A(p_j,1)^2}{p_j}+O(\frac{|A(p_j,1)|}
{p_j^{2-2\theta_3}}))\Big)\\
\leq 2^{-\frac{\ell_j}{2}}\exp\Big(\frac{k^2-1}{2}
\sum_{p_j\in P_j}(\frac{A(p_j,1)^2}{p_j}+O(\frac{|A(p_j,1)|}
{p_j^{2-2\theta_3}}))\Big).
\end{multline*}
So we get
\begin{equation}\label{eqn:S1>}
  S_1\geq \frac{X\log X}{\zeta(2-2\theta_3-0.01)^{35}} \prod_{j=1}^R\Bigg ((1-2^{-\frac{\ell_j}{2}})\exp\Big(\frac{k^2-1}
  {2}\sum_{p_j\in P_j}(\frac{A(p_j,1)^2}{p_j}+O(\frac{|A(p_j,1)|}
{p_j^{2-2\theta_3}}))\Big)\Bigg ).
\end{equation}

Now we estimate $S_2$.
Consider the primes dividing $\prod_{i=1}^R (n_i)_1$.
Let indicator function $\tilde{b}_{i,l}(n_i)=b_i(n_ip^l)$ when $b_i(p)\neq 0$ and $\tilde{b}_{i,l}(n_i)=b_i(n_i)$ otherwise, we have
\begin{equation*}
\begin{split}
S_2&\leq X\sum_{p\in \cup P_j}(\sum_{l\geq 0}\frac{A(p)^{2l+1}G_p(p^{2l+1})\log p}{p^{l+1}}\frac{(k-1)^{2l+1}}{(2l+1)!})\\
& \hskip 30pt \times
\prod_{i=1}^R (\sum_{(n_i,p)=1}\frac{A(n_i)\prod_{p\neq p_i\in R_i} G_{p_i}(n_i)}
{\sqrt{n_i(n_i)_1}}
\frac{(k-1)^{\Omega(n_i)}\tilde{b}_{i,l}(n_i)}
{g(n_i)})\prod_{p>X^{\frac{1}{\ell_R^2}}}G_p(1)\\
&=:X\prod_{p>X^{\frac{1}{\ell_R^2}}}G_p(1)
\sum_{p\in \cup P_j}S_p.
\end{split}
\end{equation*}

When $p\in P_{i_0}$, we know
\begin{equation*}
\begin{split}
S_p=&\prod_{\substack{i=1\\ i\neq i_0}}^{R}
(\sum_{n_i}\frac{A(n_i)b_i(n_i)\prod_{p_i\in P_i}G_{p_i}(n_i)}
{\sqrt{n_i(n_i)_1}}
\frac{(k-1)^{\Omega(n_i)}}{g(n_i)})
(\sum_{l\geq 0}\frac{A(p)^{2l+1}G(p^{2l+1})\log p}{p^{l+1}}\frac{(k-1)^{2l+1}}{(2l+1)!})\\
&\hskip 30pt \times
 (\sum_{(n_{i_0},p)=1}\frac{A(n_{i_0})
 \tilde{b}_{i_0,l}(n_{i_0})\prod_{p\neq p_i\in P_{i_0}}G_{p_i}(n_{i_0})}
{\sqrt{n_{i_0}(n_{i_0})_1}}
\frac{(k-1)^{\Omega(n_{i_0})}}
{g(n_{i_0})}).
\end{split}
\end{equation*}
Removing the restriction for $\tilde{b}_{i,l}$ on $\Omega(n_i)$, just as our discussion above, we have
\begin{equation*}
\begin{split}
&\sum_{(n_i,p)=1}\frac{A(n_i)b_i(n_i)
\prod_{p_i\in P_i}G_{p_i}(n_i)}
{\sqrt{n_i(n_i)_1}}
\frac{(k-1)^{\Omega(n_i)}}{g(n_i)}\\
&\leq \prod_{\substack{q\in P_i\\(p,q)=1}}\Big(\sum_{m=0}^{\infty}
\frac{A(q)^{2m}G_q(q^{2m})}{q^m}
\frac{(k-1)^{2m}}{(2m)!}
+\sum_{m=0}^{\infty}\frac{A(q)^{2m+1}G_q(q^{2m+1})} {q^{m+1}}\frac{(k-1)^{2m+1}}{(2m+1)!}\Big)\\
&\leq \exp\Big(\frac{k^2-1}{2}
\sum_{q\in P_i}(\frac{A(q,1)^2}{q}+
O(\frac{|A(q,1)|}{q^{2-2\theta_3}}))\Big),
\end{split}
\end{equation*}
and the error is no more than
\begin{multline*}
\sum_{n_j}\frac{A(n_i)\prod_{p\neq q\in P_i}G_q(n_i)}
{\sqrt{n_i(n_i)_1}}
\frac{(k-1)^{\Omega(n_i)}
2^{l\delta_{i=i_0}+\Omega(n_i)-\ell_i}}{g(n_i)}
\\
\leq 2^{l\delta_{i=i_0}-\frac{\ell_i}{2}}
\exp\Big(\frac{k^2-1}{2}
\sum_{q\in P_i}(\frac{A(q,1)^2}{q}+
O(\frac{|A(q,1)|}{q^{2-2\theta_3}}))\Big).
\end{multline*}
From $|A(p,1)|\leq 3p^{\theta_3}$ it shows that
\begin{equation*}
\begin{split}
S_p&\leq \prod_{i=1}^{R}\exp\Big(\frac{k^2-1}{2}
\sum_{q\in P_i}(\frac{A(q,1)^2}{q}+
O(\frac{|A(q,1)|}{q^{2-2\theta_3}}))\Big)\\
&\hskip 30pt\times
\prod_{\substack{i=1\\ i\neq i_0}}^{R}
(1+2^{-\frac{\ell_i}{2}})
(\sum_{l\geq 0}
\frac{A(p)^{2l+1}G(p^{2l+1})\log p}{p^{l+1}}\frac{(k-1)^{2l+1}}{(2l+1)!}
(1+2^{l-\frac{\ell_i}{2}}))\\
&\leq \prod_{i=1}^{R}\exp\Big(\frac{k^2-1}{2}
\sum_{q\in P_i}(\frac{A(q,1)^2}{q}+
O(\frac{|A(q,1)|}{q^{2-2\theta_3}}))\Big)\\
&\hskip 30pt\times
\prod_{\substack{i=1\\ i\neq i_0}}^{R}
(1+2^{-\frac{\ell_i}{2}})
(\sum_{l\geq 0}\frac{A(p)^{2l+1}G_p(p^{2l+1})\log p}{p^{l+1}}\frac{(k-1)^{2l+1}}{(2l+1)!}
(1+2^{l-\frac{\ell_i}{2}}))\\
&\leq B\frac{A(p)^2\log p}{p}\prod_{i=1}^{R}\exp \Big(\frac{k^2-1}{2}\sum_{q\in P_i}\frac{A(q,1)^2}{q}\Big),
\end{split}
\end{equation*}
where $B$ is a constant depending only on $k$.
Thus
\begin{equation}\label{eqn:S2<}
\begin{split}
S_2&\leq \zeta(2-2\theta_3-0.01)^{35}BX\prod_{i=1}^{R}\exp\Big(\frac{k^2-1}{2}
\sum_{q\in P_i}\frac{A(q,1)^2}{q}\Big)
\sum_{p\in \cup P_j}\frac{A(p,1)^2\log p}{p}\\
&\leq \zeta(2-2\theta_3-0.01)^{35}BX(\frac{\log X}{10^{2M}}+O(1))\exp\Big(\frac{k^2-1}{2}\sum_{q\in P_i}\frac{A(q,1)^2}{q}\Big).
\end{split}
\end{equation}
For $M$ large enough, by \eqref{eqn:S1>} and \eqref{eqn:S2<}, we obtain
\begin{equation*}
S_1-\frac{2}{3}S_2\gg X\log X\prod_{i=1}^{R}\exp\Big((\frac{k^2-1}{2}
)\sum_{q\in P_i}\frac{A(q,1)^2}{q}\Big)\gg X(\log X)^{\frac{k^2+1}{2}},
\end{equation*}
which proves Proposition \ref{pp1}.

\section{Proof of Proposition \ref{pp2}}

In this section we will prove Proposition \ref{pp2} with the help  of the following lemmas.

\begin{lemma}[\cite{SD00}]\label{prsum}
For any odd $n>0$, we have
\begin{equation*}
\sum_{2\nmid d}\nolimits^\flat \chi_{8d}(n)\Phi(\frac{d}{X})=
\delta_{n=\square}{\hat\Phi}(0) \frac{2X}{3\zeta(2)}\prod_{p|n}
(\frac{p}{p+1})+
O(X^{1/2+\varepsilon}\sqrt{n}),
\end{equation*}
where $\delta_{n=\square}=1$ if $n$ is a square and $\delta_{n=\square}=0$ otherwise.
\end{lemma}

\begin{lemma}\label{lmNj}
For any $1\leq j\leq R$, we have
\begin{equation*}
{\mathcal N}_j(d,k-1)^{\frac{k}{k-1}}\leq {\mathcal N}_j(d,k)\frac{(1+e^{-\ell_j})^{\frac{k}{k-1}}}
{(1-e^{-\ell_j})}+(\frac{12k^2{\mathcal P}_j(d)}{\ell_j})^{2r_k\ell_j}.
\end{equation*}
where $r_k=\lceil \frac{k}{2k-2}\rceil$.
\end{lemma}

\begin{proof}
We will prove this lemma analogous to \cite[Lemma 3.4]{GZ21} and \cite[Lemma 1]{HS20}.
For $|x|\leq \frac{K}{10}$, we have
\begin{equation*}
|\sum_{r=0}^{K}\frac{x^r}{r!}-e^x|\leq \frac{|x|^K}{K!}\leq (\frac{e}{10})^K
\leq e^{-|x|-K},
\end{equation*}
then
\begin{equation*}
\sum_{r=0}^{K}\frac{x^r}{r!}=e^x(1+O(
e^{-K}))
\end{equation*}
and the $O$-constant could takes $1$.
For $x=\alpha{\mathcal P}_j(d)$ and $K=\ell_j$, with $\alpha>0$ and $|\alpha{\mathcal P}_j(d)|\leq \ell_j /10$, we have
\begin{equation*}
{\mathcal N}_j(d,\alpha)= \exp(\alpha{\mathcal P}_j(d))(1+O(e^{-\ell_j})).
\end{equation*}
Thus when $|{\mathcal P}_j(d)|\leq \frac{\ell_j}{10k}$, we have
\begin{equation}\label{eqn:N<1}
{\mathcal N}_j(d,k-1)^{\frac{k}{k-1}}\leq
\exp(k{\mathcal P}_j(d))(1+e^{-\ell_j})^{\frac{k}{k-1}}\leq
|{\mathcal N}_j(d,k)|(1+e^{-\ell_j})^{\frac{k}{k-1}}
(1-e^{-\ell_j})^{-1}.
\end{equation}
When $|{\mathcal P}_j(d)|\geq \frac{\ell_j}{10k}$, we know
\begin{equation}\label{eqn:N<2}
|{\mathcal N}_j(d,k-1)|\leq
\sum_{r=0}^{\ell_j}\frac{|(k-1){\mathcal P}_j(d))|^r}{r!}\leq |k{\mathcal P}_j(d))|^{\ell_j}\sum_{r=0}^{\ell_j}
(\frac{10k}{\ell_j})^{\ell_j-r}\frac{1}{r!}
\leq(\frac{12k^2|{\mathcal P}_j(d))|}{\ell_j})^{\ell_j}.
\end{equation}
By \eqref{eqn:N<1} and \eqref{eqn:N<2}, we complete the proof of the lemma.
\end{proof}

From $\prod_{j=1}^{R}\max(\frac{(1+e^{-\ell_j})^{\frac{k}{k-1}}}
{(1-e^{-\ell_j})},1)\ll 1$, we have
\begin{equation}\label{eqpp2}
\begin{split}
\sum_{2\nmid d}\nolimits^\flat {\mathcal N}(d,k-1)^{\frac{k}{k-1}}
\Phi(\frac{d}{X})&\ll
\sum_{2\nmid d}\nolimits^\flat \prod_{j=1}^{R}({\mathcal N}_j(d,k)+(\frac{12k^2{\mathcal P}_j(d)}{\ell_j})^{2r_k\ell_j})
\Phi(\frac{d}{X})\\
&=:\sum_{2\nmid d}\nolimits^\flat \prod_{j=1}^{R}(\sum_{n_j\leq X^{\frac{2r_k}{\ell_j}}}\frac{a_{n_j}b_j(n_j)}
{\sqrt{n_j}}\chi_{8d}(n_j))\Phi(\frac{d}{X}).
\end{split}
\end{equation}

Recall that
\begin{equation*}
{\mathcal N}_j(d,k)=\sum_{n_j}\frac{A(n_j)}
{\sqrt{n_j}}\frac{k^{\Omega(n_j)}}
{g(n_j)}b_j(n_j)\chi_{8d}(n_j),
\end{equation*}
and
\begin{equation*}
{\mathcal P}_j(d)^{2r_k\ell_j}=\sum_{\substack{
\Omega(n_j)=2r_k\ell_j\\p\mid n_j\Rightarrow p\in P_j}}\frac{A(n_j)}{\sqrt{n_j}}
\frac{(2r_k\ell_j)!}{g(n_j)}\chi_{8d}(n_j),
\end{equation*}
where $r_k=\lceil \frac{k}{2k-2}\rceil$.
From
\begin{equation}\label{n!eq}
(\frac{n}{e})^n\leq n!\leq n(\frac{n}{e})^n,
\end{equation}
we know
\begin{equation*}
(\frac{12k^2}{\ell_j})^{2r_k\ell_j}
(2r_k\ell_j)!\leq 2r_k\ell_j(\frac{24k^2}{e})^{2r_k\ell_j}.
\end{equation*}
By setting $M$ large enough so that $\ell_R>12r_k$,
we have $|a_{n_j}|\leq B_k^{\ell_j}n_j^{\theta_3}$ with some constant $B_k$ depending only on $k$.

Using Lemma \ref{prsum} on the right side in  \eqref{eqpp2}, the contribution from error terms in Theorem \ref{prsum} are
\begin{equation*}
\ll X^{\frac{1}{2}+\varepsilon}
\prod_{j=1}^{\ell_j} B_k^{\ell_j}
X^{(1+\theta_3)\frac{2r_k}{\ell_j}}\ll
B_k^{R\ell_1}X^{\frac{1}{2}
+\frac{4(1+\theta_3)r_k}{\ell_R}+\varepsilon}
\ll X^{1-\varepsilon}.
\end{equation*}
The main term of the right side in \eqref{eqpp2} is a scalar multiple of
\begin{multline*}
  X\times\sum_{\prod_{j=1}^R n_j=\square}(\prod_{j=1}^{R}
\frac{a_{n_j}b_j(n_j)}
{\sqrt{n_j}}\prod_{p\mid n_j}\frac{p}{p+1})
=
X\times\prod_{j=1}^{R}
\sum_{n_j=\square}\frac{a_{n_j}b_j(n_j)}
{\sqrt{n_j}}\prod_{p\mid n_j}\frac{p}{p+1}
\\
=X\times\prod_{j=1}^{R}
\Big (
\sum_{n_j=\square}\frac{A(n_j)}{\sqrt{n_j}}
\frac{k^{\Omega(n_j)}}{g(n_j)}b_j(n_j)
\prod_{p\mid n_j}\frac{p}{p+1}
\\
 +(\frac{12k^2}{\ell_j})^{2r_k\ell_j}
(2r_k\ell_j)!\sum_{\substack{n_j=\square \\
\Omega(n_j)=2r_k\ell_j\\p\mid n_j\Rightarrow p\in P_j}}
\frac{A(n_j)}{\sqrt{n_j}}
\frac{1}{g(n_j)}
\prod_{p\mid n_j}\frac{p}{p+1}
\Big ).
\end{multline*}

From $|A(p,1)|\leq 3p^{\theta_3}$ and the Taylor expansion of $\exp(x)$ we can see
\begin{equation*}
\begin{split}
\sum_{n_j=\square}\frac{A(n_j)}{\sqrt{n_j}}
\frac{k^{\Omega(n_j)}}{g(n_j)}b_j(n_j)
\prod_{p\mid n_j}\frac{p}{p+1}
&\leq \prod_{p\in P_j}\Big(1+\sum_{1\leq i\leq \lceil \frac{\ell_i}{2}\rceil}\frac{A(p,1)^{2i}k^{2i}}{p^{i}(2i)!}
\frac{p}{p+1}\Big)
\\
&\leq \prod_{p\in P_j}(1+\frac{A(p,1)^2k^2}{2p}\frac{p}{p+1}
(1+\sum_{i\geq 1}\frac{(3k)^{2i}}{p^{(1-2\theta_3)i}(2i)!}))\\
&\leq \exp\Big(\frac{k^2}{2}\sum_{p\in P_j}\frac{A(p,1)^2}{p}(1+
\frac{e^{3k}}{p^{1-2\theta_3}})\Big)\\
&\ll_k \exp(\frac{k^2}{2}\sum_{p\in P_j}\frac{A(p,1)^2}{p}).
\end{split}
\end{equation*}

From \eqref{doueq} and \eqref{n!eq} we know
\begin{equation*}
\begin{split}
&(\frac{12k^2}{\ell_j})^{2r_k\ell_j}
(2r_k\ell_j)!\sum_{\substack{n_j=\square\\
\Omega(n_j)=2r_k\ell_j\\p\mid n_j\Rightarrow p\in P_j}}
\frac{A(n_j)}{\sqrt{n_j}}
\frac{1}{g(n_j)}
\prod_{p\mid n_j}\frac{p}{p+1}\\
&
\leq
(\frac{12k^2}{\ell_j})^{2r_k\ell_j}
\frac{(2r_k\ell_j)!}{(r_k\ell_j)!}
(\sum_{p\in P_j}
\frac{A(p,1)^2}{p})^{r_k\ell_j}
\\
&
\leq
2r_k\ell_j(\frac{576k^4r_k}{e\ell_j})
^{r_k\ell_j}(\sum_{p\in P_j}
\frac{A(p,1)^2}{p})^{r_k\ell_j}\\
&
\leq
2r_k\ell_j(\frac{1152k^4r_k}{eN})
^{r_k\ell_j}\\
&
\ll e^{-\ell_j}\exp\Big(\frac{k^2}{2}\sum_{p\in P_j}\frac{A(p,1)^2}{p}\Big).
\end{split}
\end{equation*}

Then we have
\begin{equation*}
\begin{split}
\sum_{2\nmid d}\nolimits^\flat {\mathcal N}(d,k-1)^{\frac{k}{k-1}}
\Phi(\frac{d}{X})&\ll
X\prod_{j=1}^{R}(1+e^{-\ell_j})
\exp\Big(\frac{k^2}{2}\sum_{p\in P_j}\frac{A(p,1)^2}{p}\Big)\ll X(\log X)^{\frac{k^2}{2}}.
\end{split}
\end{equation*}
This proves Proposition \ref{pp2}.

\section*{Acknowledgements}
The authors would like to thank referees for their careful reading and nice comments.

\end{document}